\documentclass[11pt]{article}
\usepackage[utf8]{inputenc}
\usepackage{mainsty}
\usepackage{geometry}[margin = 1in, top = 1in, bottom = 1in]
\usepackage{authblk}

\title{A Linear Variance Bound for Ballistic Deposition}
\author{Timothy Sudijono\footnote{Email: \url{tsudijon@wharton.upenn.edu}}}
\date{\today}

\begin{document}
\maketitle

\begin{abstract}
We show the height at time $t$ of the standard continuous-time ballistic deposition process has variance at most $C_d t, t \geq 1$ in every dimension $d \geq 1$. The crucial idea is to adapt a martingale argument due to Aldous, which gives tail bounds on first hitting times for monotone surface growth processes. Inverting these tail bounds and applying a duality theorem due to Penrose yields a linear variance bound for the height of ballistic deposition.
\end{abstract}



\section{Introduction}

The ballistic deposition process is a model of continuous-time surface growth widely studied in the physics literature \cite{vold1959numerical,family1985scaling,meakin1986ballistic,jullien1987simple,liu1988universality,schaaf2000random}. Consider the standard mathematical version of the model, defined as follows. Assign i.i.d. Poisson clocks over each site $x$ in the lattice $\bb{Z}^d$, which ring at rate $1$. When a clock rings at site $x$ and time $t$, the height of the surface $f(t,x)$ is updated to
\begin{equation}
\label{eq:ballistic_deposition_update}
\max\set{f(t^-,x)+1, \max_{y \sim x} f(t^-,y)},
\end{equation}
where $y \sim x$ denotes $\norm{y - x}_1 = 1$ and $t^-$ is a time infinitesimally before $t.$ The initial surface is flat.

Rigorous results concerning the exact model \eqref{eq:ballistic_deposition_update} are scant, even when $d = 1$. Sepp\"{a}l\"{a}inen \cite{seppalainen2000strong} proves a strong law of large numbers for the process. Penrose and Yukich \cite{penrose2001mathematics,penrose2002limit} establish central limit theorems on the process, while Penrose \cite{penrose2008growth} establishes the best known bounds on the size of surface height fluctuations, to be discussed shortly. Finally, Chatterjee \cite{chatterjee2021existence} demonstrates the existence of a stationary measure for the gradient process and conjectures the existence of multiple stationary measures.

For variants of the model, there has been some progress. Atar et al. \cite{atar2001ballistic} study the process on a one-dimensional strip, with extensions to general graphs analyzed in follow-up work \cite{braun2020growth, mansour2019ballistic}. Chatterjee \cite{chatterjee2023superconcentration} shows a variance bound of order $t/\log t$ for a synchronous version of the model driven by Gaussian noise. Comets et al. \cite{comets2022scaling} establish scaling limits for a variant of the model with heavy tailed block sizes and a general probability $p$ of sticking to the first point of contact. Groisman et al. \cite{groisman2025gue} demonstrate Tracy-Widom height fluctuations for the variant with one-sided sticking, near the axis.

One fundamental question on the model \eqref{eq:ballistic_deposition_update} concerns the order of fluctuations of the surface height at the origin, $f(t,0).$ It is conjectured that the model falls into the KPZ universality class \cite{kardar1986dynamic} and so $\Var f(t,0)$ is expected to scale as $t^{2/3}$ when $d = 1.$ Thus far, the only rigorous results known are that $\Var f(t,0) = \Omega(\log t)$ in one dimension and $\Var f(t,0) = o(t^2)$ in all dimensions, both due to Penrose \cite{penrose2008growth}.   

In this note, we establish an $O(t)$ upper bound on the variance, in all dimensions. 

\begin{theorem}
\label{thm:linear_variance_bound}
For each $d \geq 1$, there are constants $c_d,C_d \in (0,\infty)$ such that for every $t \geq 1, r \geq 0,$
\begin{equation}
\label{eq:bernstein_type_bound_on_height}
    \Prob\left(|f(t,0) - \E f(t,0)| \geq r \right) \leq C_d \exp\left\{ -c_d \frac{r^2}{t + r} \right\}.
\end{equation}
As a result,
\begin{equation}
\label{eq:bd_linear_variance}
    \Var f(t,0) \leq C_d t, \qquad t \geq 1.
\end{equation}
\end{theorem}

Ballistic deposition admits a last passage percolation-type representation in terms of maximal weights of lattice paths \cite{penrose2008growth}. Given such a representation, the typical strategy is to apply bounded differences or a related concentration inequality, as in \cite{sudijono2025fluctuation}. However, the approach proves unfruitful. The effect of adding an update on future surface heights is generally unbounded and appears hard to study. 

Instead, Theorem \ref{thm:linear_variance_bound} uses a different strategy, which may prove useful in analyzing variants of the model. The starting point is to use a duality result established by Penrose \cite{penrose2008growth}. Let $\xi^0(t,x)$ denote the height of a dual ballistic deposition process started from the initial configuration 
\begin{equation}
\label{eq:seed_initial_configuration}
\xi^0(0,x) = 
\begin{cases}
0, & \quad x = 0 \\
-\infty, & \quad x \neq 0.
\end{cases}
\end{equation}
Let $\eta^0(t,x)$ denote the dual next-arrival height process, which represents the new surface height at $(t,x)$ upon a hypothetical clock ring at $(t,x)$:
\begin{equation}
\label{eq:next_arrival_surface}
    \eta^0(t,x) = \max\left\{\xi^0(t,x) + 1, \max_{y \sim x} \xi^0(t,y)\right\}.
\end{equation}


Let $H_t^0 := \sup_{x \in \bb{Z}^d} \eta^0(t,x)$ denote the maximum of the dual next-arrival process.  Then, \cite[Theorem 3]{penrose2008growth} shows that the height of ballistic deposition started from a flat profile is equal in distribution to a delayed next-arrival process:
\begin{equation}
\label{eq:duality_penrose}
    f(t,0) \stackrel{(d)}{=} 
    \begin{cases}
    0, & \quad t < E \\
    H^0_{t - E}, & \quad t \geq E,
    \end{cases}
\end{equation}
for an independent exponentially-distributed time $E$ with rate $1.$ 

A helpful simplification comes by noticing that $\eta^0(t,x)$ is Markov, evolving according to the following dynamics. At a clock ring $(t,x),$
\begin{equation}
\label{eq:next_arrival_height_autonomous_dynamics}
\begin{split}
& \eta^0(t,x) = \eta^0(t^-,x) + 1 \\
    & \eta^0(t,y) = \max\set{\eta^0(t^-,x), \eta^0(t^-,y)}, \qquad y \sim x.
\end{split}
\end{equation}
The height remains unchanged at all other sites. This is true for any initial configuration of the underlying ballistic deposition. Moreover, the delayed process on the right hand side of \eqref{eq:duality_penrose} turns out to be equal in distribution to  $H_t := \sup_{x \in \bb{Z}^d} \eta(t,x)$, for a next-arrival process $\eta$ evolving autonomously according to \eqref{eq:next_arrival_height_autonomous_dynamics}, started from the seed initial configuration $\eta(0,x) = 0, x = 0$ and $-\infty$ otherwise. Therefore, we can re-express Penrose's duality result as
\begin{equation}
\label{eq:duality_reexpressed}
f(t,0) \stackrel{(d)}{=} H_t.
\end{equation}

The crucial observation is that monotonicity of the next-arrival process $\eta$ implies concentration for hitting times of its maximum, by an argument of Aldous \cite{aldous2016weak}. Let $T_n$ denote the first hitting time $T_n = \inf\set{ t \geq 0: H_t \geq n}$ and consider the remaining-time martingale
\[
M_t = \E[T_n \mid \cl{F}_t] - \E T_n,
\]
after localizing appropriately to a finite box. Using only the monotonicity and the surface-growth structure of $\eta$, we will show that the following exponential transformations of $M_t$ are supermartingales by inspecting the generator of $\eta$.
\begin{align}
    & \exp\set{\lambda M_t - \lambda^2(t \wedge T_n)/2}, \quad \lambda \geq 0 \\
    & \exp\set{-\theta M_t - (e^\theta - 1 - \theta)(t \wedge T_n)}, \quad \theta \geq 0.
\end{align}

With $\mu_n = \E T_n,$ the standard Chernoff trick yields for all $r \geq 0$,
\begin{equation}
\Prob(T_n \geq \mu_n + r) \leq \exp\set{-\frac{r^2}{2(\mu_n + r)}}.
\end{equation}
For $0 \leq r \leq \mu_n$, we also obtain the lower tail estimate
\begin{equation}
\Prob(T_n \leq \mu_n - r) \leq \exp\set{-\frac{r^2}{4\mu_n}}. 
\end{equation}
Finally, we may transfer these concentration estimates for the maximum process $H_t$ using the simple relation $\Prob(T_n \leq t) = \Prob(H_t \geq n)$, for all $n \geq 1, t \geq 0$. This inversion argument appears in \cite{penrose2008growth,sudijono2025fluctuation} when establishing variance lower bounds on surface height. Here, the important part is to ensure that the sequence of hitting time means $(\mu_n)_{n \geq 1}$ is not too ``flat", requiring some careful lower bounds on the mean differences $\mu_{n+k} - \mu_n$ for general $n,k \geq 0$. The structure of the dual process plays a role in this step. The remainder of this article gives the rigorous proof of Theorem \ref{thm:linear_variance_bound}.

\section{Preliminaries}

We fix the following notation. Locations $x,y \in \bb{Z}^d$ will often be referred to as sites, with $y \sim x$ denoting $\norm{x - y}_1 = 1$. We will let $\cl{N}_x$ denote the set of its strict neighbors $\set{y:y \sim x}.$ Clock rings over a site $x$ at time $t$ will be identified by the pair $(t,x)$. Occasionally, these will be called updates. When discussing surface growth processes $f(t,x)$, we will often write $f_t$ to denote the function $f_t(x) = f(t,x)$ on $\bb{Z}^d$.

Before embarking on the study of the next-arrival process, we derive some of its relationships with ballistic deposition. Let $\xi(t,x)$ denote the surface height of a ballistic deposition process started from any initial configuration, for now. Let $\eta(t,x)$ be its next-arrival process, defined by \eqref{eq:next_arrival_surface}. As mentioned, $\eta_t$ is Markov. 
\begin{proposition}
\label{prop:next_arrival_markov}
The next-arrival height process $\eta(t,x)$ for a ballistic deposition process started at any initial configuration evolves according to Equation \eqref{eq:next_arrival_height_autonomous_dynamics}. 
\end{proposition}
\begin{proof}
Upon a clock ring $(t,x),$ the surface height $\xi(t,x)$ is set to $\eta(t^-,x)$, by definition. Heights at other sites remain the same. Thus,
\[
\eta(t,x) = \max\set{\eta(t^-,x)+1, \max_{y \sim x} \xi(t^-,y)}.
\]
But $\eta(t^-,x) \geq \max_{y \sim x} \xi(t^-,y)$, which shows $\eta(t,x) = \eta(t^-,x) + 1.$

Upon the same clock ring $(t,x)$, we now calculate the effect on surface height at an adjacent site $y \sim x$. By \eqref{eq:next_arrival_surface} again,
\begin{align*}
\eta(t,y) = \max\set{\xi(t^-,y)+1, \ \max_{\substack{z \sim y \\ z \neq x}} \xi(t^-,z), \ \eta(t^-,x) }.
\end{align*}
Because $\eta(t^-,x) \geq \xi(t^-,x)$, we also have
\begin{align*}
\eta(t,y) & = \max\set{\xi(t^-,y)+1, \ \max_{z \sim y} \xi(t^-,z), \ \eta(t^-,x) } \\
          & = \max\set{\eta(t^-,y), \eta(t^-,x)}.
\end{align*}
\end{proof}

As in the introduction, let $\xi^0(t,x)$ be a ballistic deposition process started from the seed initial configuration \eqref{eq:seed_initial_configuration}, $\eta^0(t,x)$ its next-arrival process, and $H_t^0$ be the supremum of $\eta_t^0.$ Meanwhile, let $H_t := \sup_{x \in \bb{Z}^d} \eta(t,x)$, for $\eta$ evolving according to \eqref{eq:next_arrival_height_autonomous_dynamics}, started from the profile $\eta(0,x) = 0, x = 0$ and $-\infty$ otherwise.

\begin{lemma}[Delayed dual identity]
\label{lemma:delayed_seed_identity}
Let $E$ be exponential with rate $1$, independently of $\eta^0$. Let $\widehat H^0_t$ be the delayed process
\[
\widehat H^0_t =
\begin{cases}
    0, & \quad t < E \\
    H^0_{t - E}, & \quad t \geq E,
\end{cases}
\]
Then $(H_t)_{t \geq 0} \stackrel{(d)}{=} (\widehat H^0_t)_{t \geq 0}$ as processes.
\end{lemma}
\begin{proof}
Let $E$ be the first clock ring at the origin for $\eta$. Before $E$, every ring is at a site of height $-\infty$ and leaves the profile unchanged. At time $E$, the update rule gives
\[
\eta(E,x) =
\begin{cases}
1, & x=0,\\
0, & x\sim 0,\\
-\infty, & \text{otherwise}.
\end{cases}
\]
This is exactly $\eta^0(0,\cdot)$, as follows from \eqref{eq:seed_initial_configuration} and \eqref{eq:next_arrival_surface}. By the strong Markov property of the independent Poisson clocks, the clocks after $E$, shifted by $E$, are independent of $E$ and have their original law. The autonomous dynamics therefore identify $(\eta_{E+s})_{s \geq 0}$ with an independent copy of $(\eta^0(s,\cdot))_{s\geq0}$ in distribution. Taking maxima proves the claim.
\end{proof}

In particular, this establishes \eqref{eq:duality_reexpressed}.

\section{The Next-Arrival Process and Tail Estimates}

Given Proposition \ref{prop:next_arrival_markov}, we may study the next-arrival process without reference to an underlying ballistic deposition. Let $\Lambda \subset \bb{Z}^d$ be a finite box containing $0$. Let $\eta^\Lambda(t,x)$ denote the surface height of a next-arrival deposition process started from the seed initial configuration
\eqref{eq:seed_initial_configuration}, following the update rule 
\begin{equation}
\begin{split}
& \eta^\Lambda(t,x) = \eta^\Lambda(t^-,x) + 1 \\
    & \eta^\Lambda(t,y) = \max\set{\eta^\Lambda(t^-,x), \eta^\Lambda(t^-,y)}, \qquad y \sim x.
\end{split}
\end{equation}
We will restrict the updates to the finite box $\Lambda$. That is, all clock rings outside of $\Lambda$ are ignored and the update rule for sites on $\partial \Lambda$ ignores neighbors outside of $\Lambda$. Throughout the rest of the paper, we will often hide dependence on $\Lambda$. Since we are interested in the process $\eta^\Lambda$ at a fixed time $t$, $\Lambda$ can be taken suitably large enough such that $\eta^\Lambda_s = \eta^{\bb{Z}^d}_s$ for all $s \leq t$ with probability arbitrarily close to $1.$ Thus boundary handling causes no issues. 

Define the update maps $U_x$ which transform a surface height profile $f \in (\bb{Z}_{\geq 0} \cup \set{-\infty})^\Lambda$ according to the next-arrival dynamics at a clock ring at $x$:
\[
(U_xf)(x) = f(x) + 1, \qquad (U_xf)(y) = \max\set{f(y),f(x)} \quad (y \sim x),
\]
with all other coordinates remaining unchanged. The maps $U_x$ preserve pointwise ordering, in the sense that if $f,g$ are two profiles with $f \geq g$ pointwise, then $U_x f \geq U_x g$ for every $x.$ In addition, we define $H_t = \sup_{x \in \Lambda} \eta(t,x)$ to be the supremum of the next-arrival process, $S_t  = \set{x \in \Lambda: \eta(t,x) > -\infty}$ to be the support of the process, and $R_t  = \max_{x \in S_t} \norm{x}_1$ to be the maximum radius of the support. Notice that the surface height $\eta$ and its support $S_t$ are non-decreasing, and that $H_t$ only ever increases by $1$. \\

Penrose \cite[Lemma 4.1]{penrose2008growth} establishes an important characterization of $\eta$ in terms of maximal weights of paths on a disordered lattice. We will introduce it informally as follows. Consider $\Lambda \times \bb{R}_{\geq 0}$ with all clock rings marked on this graph, as in the right hand side of Figure \ref{fig:ballistic_deposition_path_representation}. A path $\gamma$ ending at $(t,x)$ is a piecewise-constant function $\gamma:[0,t] \rightarrow \Lambda$ with $\gamma(0) = 0, \gamma(t) = x.$ It may be visualized as traveling forwards in time starting at $(0,0)$ and ending at $(t,x)$. At an update $(t',x')$ with $t' \leq t$, it may jump laterally to a neighbor $y \sim x'$, or it may pass through the update and remain at $x'$. Let $W(\gamma)$ be the number of updates $\gamma$ passes through, which we will refer to as the \textit{weight} of the path. Updates for which $\gamma$ moves laterally do not contribute to the weight. Let $\Gamma_{t,x}$ be the set of all paths from $(0,0)$ to $(t,x)$. Then 
\begin{equation}
\label{eq:LPP_characterization}
\begin{split}
\eta(t,x) & = \sup_{\gamma \in \Gamma_{t,x}} W(\gamma).
\end{split}
\end{equation}
With this characterization, several useful bounds are immediate. The following is a slightly simpler version of the argument in \cite[Proposition 2.1]{penrose2008growth}. It can be strengthened with Chernoff bounds.

\begin{lemma}[Path bounds]
\label{lemma:path_bounds}
For every $t \geq 0$ and integers $k \geq 1$,
\begin{align}
    \max\set{\Prob(R_t \geq k), \Prob(H_t \geq k)} & \leq \exp(e(2d+1)t - k)
\end{align}
\end{lemma}
\begin{proof}
By the pathwise characterization \eqref{eq:LPP_characterization}, $R_t \geq k$ implies the existence of a path ending at $(t,x)$ for some $\norm{x}_1 \geq k$, with any weight. Similarly, the event $H_t \geq k$ implies the existence of a path $\gamma$ ending at $(t,x)$ for some $x$, with weight $W(\gamma)$ at least $k$. Both of these events imply the existence of distinct updates $\set{(s_i,x_i)}_{i=1}^k$ where $x_1 = 0, x_{i+1} \in \set{x_i}\cup\cl{N}_{x_i}$ and $s_1 \leq \dots \leq s_k \leq t$. There are at most $(2d+1)^{k-1}$ sequences of sites $(x_1,\dots,x_k)$ and the probability that all these ring times fall in $[0,t]$ is $\Prob(N \geq k),$ with $N \sim \textsf{Poi}(t)$.
Therefore,
\begin{align*}
    \max\set{\Prob(R_t \geq k), \Prob(H_t \geq k)} & \leq (2d+1)^{k-1}\Prob(N \geq k) \\
    & \leq (2d+1)^{k-1}\E\left[\binom{N}{k}\right] \\
    & \leq \frac{((2d+1)t)^k}{k!} \\
    & \leq e^{-k}\frac{(e(2d+1)t)^k}{k!} \\
    & \leq \exp(e(2d+1)t - k).
\end{align*}
The third line follows from the Poisson factorial moment identity. 
\end{proof}

\begin{figure}[t]
    \centering
    \includegraphics[width = \linewidth]{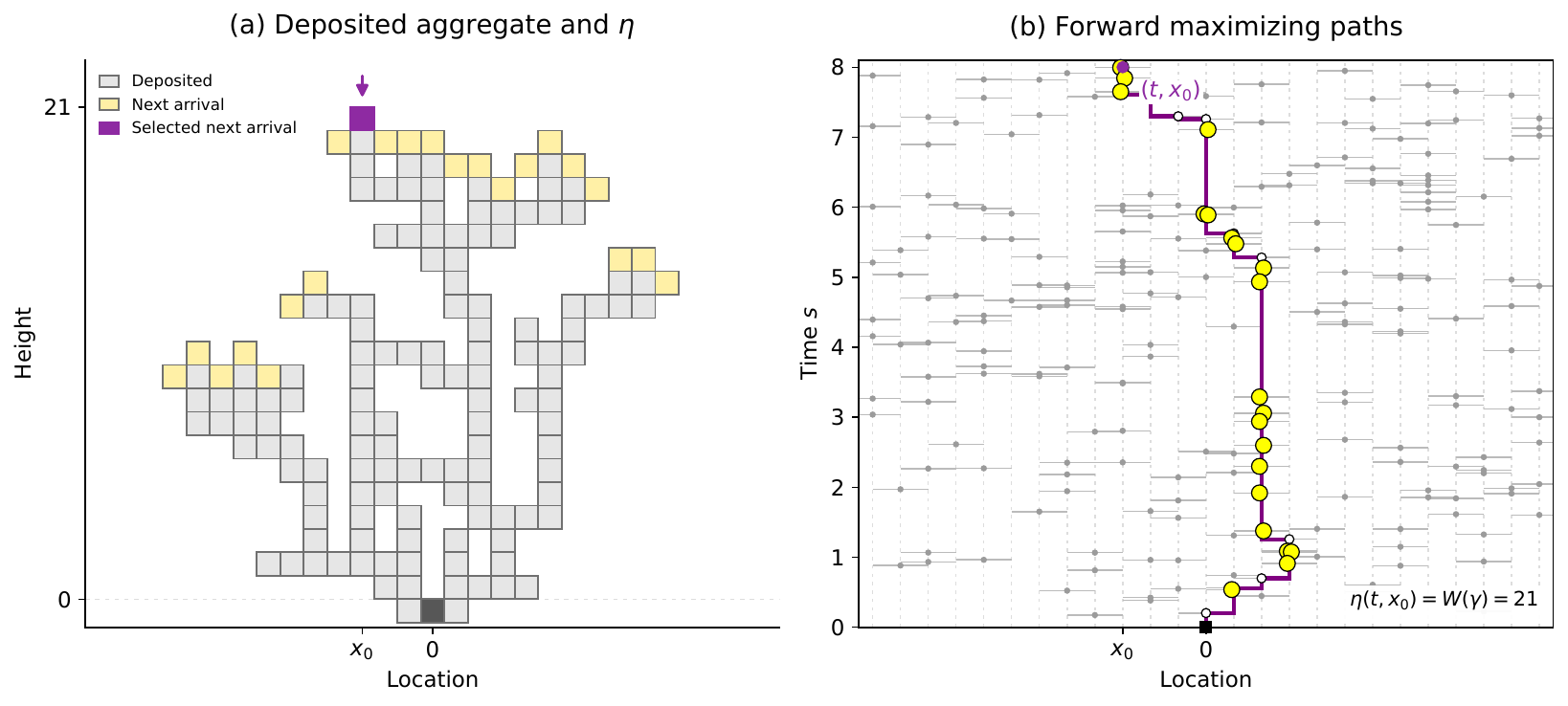}
    \caption{Visualization of the next-arrival process $\eta^0$ until $t = 8$ started from a seed, defined in terms of a ballistic deposition process $\xi^0$, along with the impact of a clock ring at location $x_0$. The grey blocks represent deposited blocks of $\xi^0.$ The right hand side shows a corresponding forward path $\gamma$ ending at $(t,x_0)$ with maximal weight. Filled yellow circles represent clock rings that $\gamma$ passes through. Locations are jittered for visibility. Note that $\eta(t,x_0)$ equals the number of filled yellow circles on $\gamma$.}
    \label{fig:ballistic_deposition_path_representation}
\end{figure}

\section{Hitting-Time Tail Estimates}
\label{sec:hitting_time_tail_estimates}

For $n \geq 1$, define the hitting times 
\begin{equation}
    T_n = \inf\set{t: H_t \geq n}, \qquad \mu_n = \E T_n, \qquad T_0 = \mu_0 = 0.
\end{equation}

With an abuse of notation, let $\mu_n(f)$ denote the expected time next-arrival ballistic deposition takes to reach maximum height at least $n$, started from a profile $f \in (\bb{Z}_{\geq 0} \cup \set{-\infty})^\Lambda$, with at least one site having nonnegative height. Choose a site $x$ with $f(x) \geq 0$. After $n$ clock rings at $x$, the height is at least $n$. So, $\mu_n(f) \leq n$. Because the update maps $U_x$ respect pointwise order of profiles,
\[
\delta_x(f) := \mu_n(f) - \mu_n(U_x f) \geq 0.
\]
So long as $T_n$ has not been reached, conditioning on the site of the first clock ring in $\Lambda$ gives
\[
\mu_n(f) = \frac{1}{|\Lambda|} + \frac{1}{|\Lambda|} \sum_{x \in \Lambda} \mu_n(U_x f).
\]
Therefore,
\begin{equation}
\label{eq:jumps_simplex}
    \sum_{x \in \Lambda} \delta_x(f) = 1, \qquad 0 \leq \delta_x(f) \leq 1.
\end{equation}
 Let $\cl{F}_t$ denote the filtration generated by the clocks up to time $t$. Following \cite[Proof of Lemma 1.1]{aldous2016weak}, consider the remaining-time martingale
\begin{equation}
\begin{split}
M_t & := \E[T_n \mid \cl{F}_t] - \mu_n.
\end{split}
\end{equation}
By the Markov property, $M_t$ is also equal to 
\[
t \wedge T_n + \mu_n(\eta_{t \wedge T_n}) - \mu_n.
\]
For $t \geq T_n$, the martingale is constant. For $t < T_n$, the martingale increases at unit speed between clock rings, with jumps $-\delta_x(\eta_{t^- \wedge T_n})$ upon a clock ring at $(t,x)$. 

Considering exponential transformations of $M_t$ and applying the crucial bound \eqref{eq:jumps_simplex} yields supermartingales which we can use for Chernoff bounds. The resulting bound is very general and holds for essentially any continuous-time surface growth process whose update maps $U_x$ respect pointwise order.

\begin{proposition}[Exponential Concentration of Hitting Times]
\label{prop:exponential_concentration_hitting_times}
For all $r \geq 0$,
\begin{equation}
\label{eq:hitting_time_concentration_upper_tail}
\Prob(T_n \geq \mu_n + r) \leq \exp\set{-\frac{r^2}{2(\mu_n + r)}}.
\end{equation}
For $0 \leq r \leq \mu_n$, we also obtain the lower tail estimate
\begin{equation}
\label{eq:hitting_time_concentration_lower_tail}
\Prob(T_n \leq \mu_n - r) \leq \exp\set{-\frac{r^2}{4\mu_n}}. 
\end{equation}
\end{proposition}
\begin{proof}
For any functional of surface heights $F:(\set{-\infty} \cup \bb{N}_{\geq 0})^\Lambda \rightarrow \bb{R}$, the generator $\cl{L}_t$ of the next-arrival process $\eta_{t}$ is given by
\begin{equation}
(\cl{L}_t F)(\eta) = \sum_{x \in \Lambda} \set{F(U_x \eta) - F(\eta)}
\end{equation}
It is a standard fact that if $X_t$ is a continuous-time Markov jump process with generator $\cl{L}_t$ and state space $S$, for any $F: \bR_{t \geq 0} \times S$ which is $C^1$ in time and bounded on finite time intervals,
\begin{equation}
    F(t,X_t) - \int_0^t \left( \partial_s F + \cl{L}_s F\right)(s,X_s) ds, \qquad t \geq 0
\end{equation}
is a martingale. See \cite[\S5.1]{eberle2020markov} for a concrete reference. For any $\lambda \geq 0$, apply this to $\eta_{t}$ with 
\[
F(t,\eta) := \exp\set{\lambda(\mu_n(\eta) - \mu_n) + \left(\lambda - \frac{\lambda^2}{2}\right)t}.
\]
Then, $\partial_s F(s,\eta_s) = (\lambda - \lambda^2/2) F$ and 
\begin{align*}
(\cl{L}_s F_s)(\eta) & = \sum_{x \in \Lambda} \set{\exp(\lambda (\mu_n(U_x \eta) - \mu_n) + (\lambda - \lambda^2/2)s) - F(s,\eta)} \\
& = F(s,\eta)\sum_{x \in \Lambda} \set{\exp(-\lambda \delta_x(\eta)) - 1}.
\end{align*}
For $s \leq t \wedge T_n$, by \eqref{eq:jumps_simplex} and the elementary inequality $e^{-x} - 1 + x \leq x^2/2, \ x \geq 0$, 
\begin{align*}
    (\partial_s F + \cl{L}_s F)(s,\eta_s) & = F(s,\eta_s)\left(\lambda - \lambda^2/2 + \sum_{x \in \Lambda} \set{\exp(-\lambda \delta_x(\eta_s)) - 1} \right) \\
    & = F(s,\eta_s)\left( - \lambda^2/2 + \sum_{x \in \Lambda} \set{\exp(-\lambda \delta_x(\eta_s)) + \lambda \delta_x(\eta_s) - 1} \right) \\
    & \leq F(s,\eta_s)\left(-\lambda^2/2 + \frac{\lambda^2}{2}\sum_{x \in \Lambda} \delta_x(\eta_s)^2  \right) \\
    & \leq 0.
\end{align*}
So $\int_0^{t \wedge T_n} \left( \partial_s F + \cl{L}_s F\right)(s,\eta_s) ds$ is a non-increasing process adapted to $\cl{F}_t$. Therefore, the stopped process $F(t \wedge T_n,\eta_{t \wedge T_n})$ is a supermartingale. Expressing this in terms of $M_t$ and rescaling shows that 
\begin{equation}
\exp\set{\lambda M_t - \frac{\lambda^2}{2}(t \wedge T_n)}
\end{equation}
is also a supermartingale starting at $1$. Therefore,
$\E[e^{\lambda M_t}] \leq e^{\frac{\lambda^2t}{2}}.$ Applying the standard Chernoff bound,
\begin{align*}
    \Prob(T_n \geq \mu_n + r) & \leq \Prob(M_{\mu_n + r} \geq r) \\
    & \leq e^{-\lambda r} \E[e^{\lambda M_{\mu_n + r}}] \\
    & \leq e^{-\lambda r} e^{\lambda^2 (\mu_n + r)/2}.
\end{align*}
Optimizing $\lambda$ gives \eqref{eq:hitting_time_concentration_upper_tail}. For the lower tail, apply the same argument with
\[
G(t,\eta) := \exp\set{-\theta(\mu_n(\eta) - \mu_n) + \left(1 - e^\theta\right)t}
\]
for any $\theta \geq 0.$ We compute
\begin{align*}
    (\partial_s G + \cl{L}_s G)(s,\eta_s) & = G(s,\eta_s)\left( 1 - e^\theta + \sum_{x \in \Lambda}\set{\exp(\theta \delta_x(\eta_s)) - 1} \right) \\
    & = G(s,\eta_s)\left( 1 + \theta - e^\theta + \sum_{x \in \Lambda} \set{\exp(\theta \delta_x(\eta_s)) -\delta_x(\eta_s) \theta - 1} \right).
\end{align*}
By the convexity of $f(x) = e^x - x - 1$ and the fact that $f(0) = 0$, we have $f(\alpha x) \leq \alpha f(x)$ for any $\alpha \in [0,1]$. Then for $s \leq t \wedge T_n$,
\begin{align*}
\sum_{x \in \Lambda} \set{\exp(\theta \delta_x(\eta_s)) -\delta_x(\eta_s) \theta - 1}  & \leq (e^{\theta} - \theta - 1) \sum_{x \in \Lambda} \delta_x(\eta_s) \\
& = (e^{\theta} - \theta - 1).
\end{align*}
By the same argument as in the first part of the proof, $G(t \wedge T_n,\eta_{t \wedge T_n})$ is a supermartingale. In particular, so is
\[
\exp\set{-\theta M_t - (e^\theta - 1 - \theta)(t \wedge T_n)}.
\]
Since $t \wedge T_n \leq t, \E[e^{-\theta M_t}] \leq e^{t(e^\theta - 1 - \theta)} \leq e^{t \theta^2},$ for $\theta \in [0,1]$. Applying the Chernoff bound and setting $\theta = r/(2\mu_n)$ yields \eqref{eq:hitting_time_concentration_lower_tail}.

\end{proof}

\section{Proof of Theorem \ref{thm:linear_variance_bound}}

When inverting the concentration bounds from the previous section, the important model-specific input is to show that the mean differences $\mu_{n+k} - \mu_n$ are not too small. It is known that $\mu_n/n \rightarrow \rho$ for some constant $\rho$, by a subadditivity argument \cite{penrose2008growth}. However, the resulting conclusion $\mu_n = \rho_n + o(n)$ is too coarse. The key is to notice that the stopped next-arrival process $\eta_{T_n}$ is bounded pointwise by $n$, because $H_t$ only increases by $1.$ Therefore, one can speed up the process by immediately jumping all heights in the support to $n$, at time $T_n$. This gives a lower bound on $T_{n+k} - T_n$ in terms of the expected hitting time of level $k$ of a next-arrival process started from a flat initial configuration supported on $S_{T_n}$. There are some resemblances to the proof of superadditivity in \cite{sudijono2025fluctuation}. 

\begin{lemma}[Hitting Time Mean Spacings]
\label{lemma:mean_spacings}
There are constants $C_d \geq 1, c_d > 0$ independent of $\Lambda$ such that with
$L_n := \log(2C_d) + d\log(n+1)$, 
\begin{equation}
\label{eq:mu_spacing_lower_bound}
    \mu_{n+k} - \mu_n \geq c_d \left(k - L_n \right).
\end{equation}
for all integers $n \geq 0,k\geq 1$. Also, $0 \leq \mu_{n+1} - \mu_n \leq 1$.
\end{lemma}
\begin{proof}
Throughout this proof, let $C,c$ be constants depending on $d$, but independent of $\Lambda$, that may change line by line. We will first establish the following bound on the size of the support at $T_n$:
\begin{equation}
\label{eq:expected_support_size}
    \E[|S_{T_n}|] \leq C(n+1)^d.
\end{equation}
Let $\kappa := \frac{1}{2(2d+1)e}$. Monotonicity of $R_t$ shows
\begin{align*}
    \Prob(R_{T_n} \geq r) & \leq \Prob(T_n >  \kappa r) + \Prob(R_{\kappa r} \geq r).
\end{align*}
To bound the first term on the right hand side, let $\Gamma_n$ denote the time of the $n$th clock ring at the origin. Note $\Gamma_n \sim \textsf{Gamma}(n,1)$ and $T_n \leq \Gamma_n$. Then $\Prob(T_n > \kappa r) \leq \Prob(\Gamma_n > \kappa r) \leq \E[e^{\Gamma_n/2}]e^{-\kappa r/2} = 2^ne^{-\kappa r / 2}$, where the last equality follows by the formula for the moment generating function of $\textsf{Gamma}(n,1)$. For the second term, apply Lemma \ref{lemma:path_bounds}. Then
\[
\Prob(R_{T_n} \geq r) \leq 2^n e^{-\kappa r / 2} + e^{-r/2}.
\]
For $r \geq C(n+1)$ for $C$ large enough, the right hand side is less than $2e^{-cr}$, for another constant $c.$ Then
\begin{align*}
    \E|S_{T_n}| & \leq \E(2R_{T_n} + 1)^d \\
    & \leq C \left[1 + \sum_{r = 1}^\infty r^{d-1}\Prob(R_{T_n} \geq r) \right] \\
    & \leq C(n+1)^d,
\end{align*}
by splitting the sum into the terms $r \geq C(n+1)$ and $r \leq C(n+1)$. \\

Next, we prove \eqref{eq:mu_spacing_lower_bound}. Consider a process $\tilde \eta_t$ as follows: $\tilde \eta_t$ is equal to $\eta_t$ for all $t < T_n$. At $T_n$, immediately set all heights to $n$ for all sites $x \in S_{T_n}$ in the support. Continue evolving $\tilde \eta_t$ according to the next-arrival rule. Let $\tilde T_{n+k}$ denote the first time $s \geq T_n$ such that $\sup_{x \in \Lambda } \tilde \eta(s,x) = n + k.$ Both processes $\tilde \eta, \eta$ evolve according to the same clocks. Because all sites in $\eta_{T_n}$ have height at most $n$, this coupling forces $\tilde{T}_{n+k} - T_n \leq T_{n+k} - T_n$, so 
\[
\Prob(T_{n+k} - T_n \leq s) \leq \Prob(\tilde T_{n+k} - T_n \leq s)
\]
for all $s \geq 0.$ 

By the strong Markov property, $\tilde{T}_{n+k} - T_n$ has the same law, conditional on $\cl{F}_{T_n}$, as the first time a next-arrival ballistic deposition process has maximum height $k$, started from the initial profile 
\[
\begin{cases}
0 & \ x \in S_{T_n} \\
-\infty & \ \text{otherwise}.
\end{cases}
\]
Note that $\set{\tilde T_{n+k} - T_n \leq s}$ implies the existence of a forwards path starting at some site $x$ in the support $S_{T_n}$, ending at time $s$, with at least $k$ many clock rings.  Following the proof of Lemma \ref{lemma:path_bounds} and taking a union bound,

\[
\Prob(\tilde T_{n+k} - T_n \leq s \mid \cl{F}_{T_n}) \leq |S_{T_n}| e^{(2d+1)es - k}.
\]

By \eqref{eq:expected_support_size},
\begin{align*}
\Prob(\tilde T_{n+k} - T_n \leq s) & \leq \E[|S_{T_n}|] e^{(2d+1)es - k} \\   
& \leq C(n+1)^d e^{(2d+1)es - k}.
\end{align*}

Take $k > L_n$. If not, then monotonicity $\mu_{n+k} \geq \mu_n$ gives \eqref{eq:mu_spacing_lower_bound}. With $L_n = \log(2C) + d\log(n+1)$, taking $s = (k - L_n)/(e(2d+1))$ gives $\Prob(\tilde T_{n+k} - T_n \leq s) \leq 1/2.$ Thus
\[
\mu_{n+k} - \mu_n \geq s\Prob(\tilde T_{n+k} - T_n \geq s) \geq \frac{k - L_n}{2e(2d+1)}.
\]
This establishes \eqref{eq:mu_spacing_lower_bound}. For the last claim, $\mu_{n+1} \geq \mu_n$ holds by monotonicity and $\mu_{n+1} \leq \mu_n + 1$ since it takes expected time $1$ for a clock to ring over a height-maximizing site of $\eta_{T_n}.$    
\end{proof}

\begin{proof}[Proof of Theorem \ref{thm:linear_variance_bound}]
Throughout this proof, again let $C,c$ be constants that may change line by line, which are independent of $\Lambda$. It suffices to prove Equation \eqref{eq:bernstein_type_bound_on_height}, as integrating yields the linear variance bound \eqref{eq:bd_linear_variance}. We claim that for every $t \geq 1$ and $r \geq 0$,
\begin{equation}
\label{eq:concentration_on_H_t}
    \Prob\left(|H_t - \E H_t| \geq r\right) \leq C \exp\set{-c\frac{r^2}{t+r}}.
\end{equation}
The constants do not depend on the box $\Lambda$. For $t \geq 1$, let
\begin{equation}
m := \min\set{n \geq 1 : \mu_n \geq t}.
\end{equation}
We record the following facts. First, $t \leq \mu_m$ by definition and $\mu_m - \mu_{m-1} \leq 1 \Rightarrow \mu_m \leq t + 1$, so $t \leq \mu_m < t+1$. Second, $m \leq Ct$ because of the spacing lower bound in Lemma \ref{lemma:mean_spacings}. Third, $L_m \leq C\log(t+2)$ because of the second fact.

For any integer $r \geq C \log(t+2),$ Lemma \ref{lemma:mean_spacings} shows $\mu_{m+r} - t \geq \mu_{m+r} -\mu_m \geq c(r - L_m) \geq cr/2$. Applying Proposition \ref{prop:exponential_concentration_hitting_times},
\begin{align*}
\Prob(H_t \geq m + r) & = \Prob(T_{m+r} \leq t) \\
                          & \leq \Prob(T_{m+r} - \mu_{m+r} \leq -cr/2) \\
                          & \leq \exp\set{-c\frac{r^2}{\mu_{m+r}}}.
\end{align*}
With $\mu_{m+r} \leq m+r \leq Ct+r,$ we obtain
\[
\Prob(H_t \geq m + r) \leq \exp\set{-c\frac{r^2}{t + r}}.
\]
For the other direction, let $r \leq m$. Using Lemma \ref{lemma:mean_spacings}, $\mu_m - \mu_{m - r + 1} \geq c_d(r-1-L_{m-r+1})$. Since $L_{m - r + 1} \leq L_m$, 
\begin{align*}
t - \mu_{m - r + 1} & = t - \mu_m + \mu_m - \mu_{m - r + 1} \\
                    & \geq -1 + c(r - 1 - L_m) \\
                    & \geq cr/2,
\end{align*}
for $r \geq C\log(t+2)$. Applying \eqref{eq:hitting_time_concentration_upper_tail}, we obtain
\begin{align*}
\Prob(H_t \leq m - r) & = \Prob(T_{m-r+1} > t) \\
                          & \leq \Prob(T_{m-r+1} - \mu_{m-r+1} > cr/2) \\
                          & \leq \exp\set{-c\frac{r^2}{r + \mu_{m-r+1}}} \\
                          & \leq \exp\set{-c\frac{r^2}{r + \mu_{m}}}.
\end{align*}

Combining these two estimates, we obtain
\begin{equation}
    \Prob(|H_t - m| \geq r) \leq C\exp\set{-c\frac{r^2}{t+r}},
\end{equation}
for all integer $r \geq C \log (t+2).$ The same bound for all $r \geq 0$ holds upon enlarging $C$. Integration by parts shows $\E|H_t - m| \leq C\sqrt{t}$. Combining this with the triangle inequality gives \eqref{eq:concentration_on_H_t}, upon potentially changing $C,c$. \\

Next, we pass to infinite volume. Let $\Lambda_L=[-L,L]^d\cap\bb{Z}^d$, and couple the processes $\eta_{t}^{\bb{Z}^d}$ and $\set{\eta^{\Lambda_L}}_{L \geq 1}$ using the same clocks, extending their profiles by $-\infty$ outside $\Lambda_L$. Clearly, the surface heights $\eta_t^{\Lambda_L}$, supports $S_t^{\Lambda_L}$, and maxima $H_t^{\Lambda_L}$ increase with $L$. By Lemma \ref{lemma:path_bounds}, 
\[
\Prob(R_t^{\bb{Z}^d} \geq r)\leq\exp\set{e(2d+1)t-r}.
\]
On $\set{R_t^{\bb{Z}^d} < L}$, all clock rings up to time $t$ which update the surface height occur in $\Lambda_L$. Therefore,
\begin{equation}
\label{eq:infinite_volume_coupling}
\Prob\left(\eta^{\Lambda_L}_s \neq \eta_s^{\bb{Z}^d} \text{ for some }s\leq t\right)
\leq\exp\set{e(2d+1)t-L}.
\end{equation}

In particular, $H_t^{\Lambda_L} \rightarrow H_t^{\bb{Z}^d}$ almost surely. Monotone convergence implies $\E H_t^{\Lambda_L} \rightarrow \E H_t^{\bb{Z}^d}$. Let $B := \set{\exists s \leq t : \eta^{\Lambda_L}_s \neq \eta_s^{\bb{Z}^d}}$. Take $L$ large enough so that $\Prob(B) \leq \e, \E H_t^{\bb{Z}^d} - \E H_t^{\Lambda_L} \leq \e$. Then 
\begin{align*}
    \Prob\left( \left| H_t^{\bb{Z}^d} - \E H_t^{\bb{Z}^d}\right| \geq r \right) & \leq \Prob\left( \left| H_t^{\bb{Z}^d} - \E H_t^{\Lambda_L} \right| \geq r - \e \right) \\
    & \leq \e + C\exp\set{-c\frac{(r-\e)^2}{t + r - \e}}.
\end{align*}
Taking $\e \downarrow 0$ establishes \eqref{eq:concentration_on_H_t} when $\Lambda = \bb{Z}^d$. Finally, using the duality \eqref{eq:duality_reexpressed} gives the result.

\end{proof}

\noindent \textbf{Acknowledgements.} This article was prepared with assistance from ChatGPT 6, which suggested several core ideas in the proof. The author wrote the exposition and prose. We take responsibility for all errors.

\bibliographystyle{unsrt}
\bibliography{main}

\end{document}